\documentclass[11pt]{amsart}

\usepackage{amsmath, amssymb,  mathabx, amsfonts,enumerate}
\usepackage[all]{xy}
\usepackage{amscd,stmaryrd}
\usepackage{comment}
\usepackage{mathtools}
\usepackage{tikz} 
\usepackage{tikz-cd} 
\tikzcdset{diagrams={nodes={inner sep=7.5pt}}}

\usepackage{url}
\usepackage{hyperref}

\newtheorem{theorem}{Theorem}[section]
\newtheorem{lemma}[theorem]{Lemma}
\newtheorem{corollary}[theorem]{Corollary}

 \theoremstyle{definition}
 \newtheorem{definition}[theorem]{Definition}

 \newtheorem{example}[theorem]{Example}

  \newtheorem*{example*}{Example}

\numberwithin{equation}{section}

\newcommand {\Z}{\mathbb{Z}} 
\newcommand {\R}{\mathbb{R}} 
 
\newcommand {\C}{\mathbb{C}}

\newcommand{\NN}{\mathcal{N}}

\newcommand{\PP}{\mathcal{P}}

\DeclareMathOperator{\Fix}{Fix}

\DeclareMathOperator{\Id}{Id}

\begin{document}
\title[Stable properties of post-injunctive groups]{Around Gromov's injectivity lemma and applications to post-injunctive groups}  
\author[Xuan Kien Phung]{Xuan Kien Phung}
\address{Département d'informatique et de recherche opérationnelle,  Université de Montréal, Montréal, Québec, H3T 1J4, Canada.}
\email{phungxuankien1@gmail.com}   

\begin{abstract}
 Gottschalk's surjunctivity conjecture states that for all group universes and finite alphabets, every equivariant and continuous  selfmap  of the full shift, known as cellular automaton, cannot be a strict embedding.   Not all surjective cellular automata are injective. However, if the  
surjectivity condition is replaced by a certain strengthened property called post-surjectivity then all post-surjective cellular automata must be bijective whenever the universe is a sofic group. A group universe is said to be post-injunctive if every post-surjective cellular automaton with finite alphabet over this group universe must be  bijective. Gromov's injectivity lemma states  each injective cellular automaton  over a subshift  can be extended to an injective cellular automaton over every subshift  which is close enough to the initial subshift.  
In this paper, we obtain analogous results where injectivity is replaced by other fundamental dynamical properties namely post-surjectivity and pre-injectivity. We also study various stable properties of the class of  post-injunctive groups in parallel to properties of surjunctive groups. Among the  results, we show that semidirect extensions of post-injunctive groups with residually finite kernels must be post-injunctive. 
\end{abstract}
\maketitle
  
\setcounter{tocdepth}{1}

\section{Introduction}

\label{s:definitions}
To state the results, we first recall some basic  notions of symbolic dynamics. 
Given a discrete set $A$ and a group $G$, the \emph{full shift} $A^G$ consists of \emph{configurations} $x \in A^G$ which are  maps $x \colon G \to A$. In particular, a {constant configuration} $c \in A^G$ refers simply to  a constant map $c\colon G \to A$. 
We say that two configurations $x,y  \in A^G$ are \emph{asymptotic} if  $x\vert_{G \setminus E}=y\vert_{G \setminus E}$ for some finite subset $E \subset G$.  A configuration $x\in A^G$ is \emph{asymptotically constant} if $x$ is asymptotic to some constant configuration in $A^G$. The \emph{Bernoulli shift} action on the full shift $G \times A^G \to A^G$ is defined by $(g,x) \mapsto g x$, 
where $(gx)(h) =  x(g^{-1}h)$ for all   $g,h \in G$,  $x \in A^G$. The {full shift} $A^G$ is equipped with the prodiscrete topology. 
For every subset $X\subset A^G$, we denote the restriction of $X$ to a subset  $F\subset G$ by  $X_F=\{x\vert_F \colon x \in X\}$. 
\par 
Following the first construction of a cellular automaton  over $\Z^2$ by von Neumann and Ulam  
  \cite{neumann} (see also  \cite{burks}),
a cellular automaton  over the group $G$ (called the \emph{universe}) and the set $A$ (called the \emph{alphabet}) is   a $G$-equivariant and uniformly continuous  self-map $A^G\to A^G$.    Equivalently, we have the following more specific definition of cellular automata thanks to a characterization of Hedlund \cite{hedlund-csc}, \cite{hedlund}. 

\begin{definition}
\label{d:ca}
Let  $G$ be a group universe and let $A$ be an alphabet. Let $M$ be a finite subset of $G$ and let $\mu \colon A^M \to A$ be a map. We define the cellular automaton $\tau \colon A^G \to A^G$  by  setting 
\begin{align*}
    \tau(x)(g)=  
 \mu((g^{-1}x)  
	\vert_M) \quad  \text{ for all } x \in A^G \text{  and } g \in G. 
\end{align*}
Such a set $M$ is called a \emph{memory} and $\mu$ is called a \emph{local transition map} of $\tau$. 
 \end{definition} 
\par 
We can think of each element $g \in G$ as a cell of the universe. Note that every cellular automaton is uniform in the sense that all the cells follow the same local transition map. 
When  different cells can evolve according to different local transition maps, we obtain a more general class of machines called non-uniform cellular automata  \cite{Den-12a}, \cite{Den-12b}, \cite[Definition~1.1]{phung-tcs}. 
\par 
Cellular automata, also called tessellation structures, are fundamental discrete models of computation in computer science with applications notably in   physics, biology, and cryptography. Their studies have also led to the discovery of various deep connections with areas of mathematics such as geometric group theory and ring theory. Two milstones results in the theory of cellular automata relating their fundamental properties are  the Garden of Eden theorem and the surjunctivity of sofic groups-a very large class of groups of widespread interests originally  introduced by Gromov to tackle Gottschalks' surjunctivity conjecture. 
We recall these basic properties of cellular automata. Let $\tau \colon A^G\to A^G$ be a cellular automaton. 
We say that  
\begin{itemize}
    \item $\tau$ is \emph{pre-injective} if $\tau(x) = \tau(y)$ implies $x= y$ whenever $x, y \in A^G$ are asymptotic configurations; 
    \item 
    $\tau$ is \emph{post-surjective} if for all $x, y \in A^G$ with $y$ asymptotic to $\tau(x)$, then   $y= \tau(z)$ for some $z \in A^G$ asymptotic to $x$;
    \item 
    $\tau$ is   \emph{reversible} or \emph{left-invertible} if there exists a cellular automaton $\sigma \colon A^G \to A^G$ such that $\sigma \circ \tau = \Id_{A^G}$;
    \item 
    $\tau$ is  \emph{invertible} if it is bijective and the inverse map  $\tau^{-1}$ is a cellular automaton. 
\end{itemize}
  Note that by the closed image property,  post-surjectivity implies surjectivity for cellular automata. The Garden of Eden theorem \cite{myhill}, \cite{moore}, \cite{tullio}  states that for cellular automata over amenable group universes, pre-injectivity and surjectivity are equivalent properties.  Gottschalks' surjunctivivty conjecture \cite{gottschalk} states that for cellular automata with finite alphabets over group universes, we have   injectivity$\implies$surjectivity. Amenable groups are precisely groups satisfying the Garden of Eden theorem \cite{bartholdi-kielak}. It is not known whether the surjunctivity conjecture holds true for every group but Gromov \cite{gromov-esav} and Weiss \cite{weiss-sgds} show that sofic groups satisfy the surjunctivity conjecture.  
A certain dual surjunctivity  version of Gottschalk's conjecture was  studied  by 
Capobianco, Kari, and Taati in \cite{kari-post-surjective} which  states that  every post-surjective cellular automaton over a group universe and a finite alphabet is also pre-injective. The authors settled in the same paper \cite{kari-post-surjective} the case of sofic group universes.  
The above results motivate the following notion of \emph{surjunctive} and \emph{post-injunctive} groups. 
\begin{definition}
    \label{def:post-injunctive} A group $G$ is  {\textit{post-injunctive}} if for every finite alphabet $A$, every post-surjective cellular automaton $\tau\colon A^G\to A^G$ must be pre-injective. Similarly,  a group $G$ is  {\textit{surjunctive}} if for every finite alphabet $A$, every injective  cellular automaton $\tau\colon A^G\to A^G$ must be surjective.
\end{definition}

In particular, every sofic group is both  surjunctive and post-injunctive. Remark that when restricted to the class of linear cellular automata, the notions of post-injunctivity and surjunctivity are in fact dual to each other. 
\par 
Surjunctive groups are known to satisfy some stable properties. Among these, Gromov's injectivity lemma (cf. \cite[Lemma 4.H"]{gromov-esav} and \cite[Theorem~3.6.1]{csc-book}) states that every injective cellular automaton $\tau \colon \Sigma \to \Sigma$ over a subshift $\Sigma\subset A^G$ can be extended to an embedding of every subshift $X\subset A^G$ which is close enough to the initial subshift $\Sigma$. Note that Gromov's injectivity lemma also provides the key argument in the proof of the closedness of the space of $\Gamma$-marked surjunctive groups.  
\par 
The main goal of this paper is to establish  analogues of Gromov's injectivity lemma for post-surjectivity and pre-injectivity as well as  several stable group-theoretic properties of post-injunctive groups that were known to be satisfied by surjunctive groups. We can summarize our main results as follows: 
\begin{itemize}
    \item Virtually post-injunctive groups are post-injunctive (Lemma~\ref{l:virtual-post-injunctive}). 
   \item  Every subgroup of a post-injunctive group is also post-injunctive (Theorem~\ref{t:subgroup-post-injunctive}). 
   \item 
   A group is post-injunctive if and only if it is locally post-injunctive (Theorem~\ref{t:locally-post-injunctive}). 
   \item 
   Fully residually post-injunctive groups are post-injunctive (Theorem~\ref{t:fully-residually-post-injunctive}). 
   \item 
 Semidirect extensions of a post-injunctive group with
finitely generated residually a finite kernel are also  post-injunctive (Theorem~\ref{t:semdirect-extension-res-finite-kernel}). 
\item 
Limits of certain sequences of pre-injective cellular automata are also pre-injective cellular automata (Lemma~\ref{l:pre-injective-mark-group}).   
\item An analogue of Gromov's injectivity lemma for post-surjectivity holds true (Lemma~\ref{l:post-surjective-mark-group}): being post-surjective is essentially an open property. 
\item The space of $\Gamma$-marked post-injunctive groups is closed for every given group $\Gamma$ (Theorem~\ref{t:marked-groups-closed}).  
\end{itemize}

The paper is organized as follows. We recall basic constructions of induced cellular automata over quotient groups in Section~\ref{s:map-phi-psi}. Well-known descriptions of the uniform structures on the space of marked groups and the space of subsets of a full shift are given in Section~\ref{s:marked-groups}. We then establish some preliminary results in Section~\ref{s:post-injunctive-groups} on post-injunctive groups. The proofs of Lemma~\ref{l:virtual-post-injunctive},  
  Theorem~\ref{t:subgroup-post-injunctive},  Theorem~\ref{t:locally-post-injunctive}, and  (Theorem~\ref{t:fully-residually-post-injunctive}) 
are given in Section~\ref{s:fully-res-locally-post-injunctive}. The proofs of Theorem~\ref{t:semdirect-extension-res-finite-kernel} and Theorem~\ref{t:marked-groups-closed} are presented in Section~\ref{s:mark-groups-proof} and Section~\ref{s:closedness-mark-groups} respectively. Analogues of Gromov's injectivity lemma for cellular automata are obtained in Section~\ref{s:analogues-gromov-lemma} for several fundamental properties such as pre-injectivity and post-surjectivity. A counterexample is also described in Section~\ref{s:analogues-gromov-lemma}.

\section{Restriction and cellular automata over quotient groups}
\label{s:map-phi-psi}
 Let $M\subset G$ be a finite memory set of a cellular automaton $\tau\colon A^G\to A^G$ and let $\mu \colon A^M \to A$ be the corresponding local transition map. For every group homomorphism $\varphi \colon G\to K$, we can construct an induced cellular automaton   $\tau_K \colon A^K \to A^K$ as follows. Let $H = \ker \varphi \subset G$ and let 
 $$\Fix(H) = \{x\in A^G\colon hx=x \text{ for all } h \in H\}$$ be the set of $H$-periodic configurations in $A^G$. \vspace{0.1cm}
\par 
\noindent 
\textbf{The auxiliary maps $\Phi$ and $\Psi$.} We have a canonical bijection: 
    \begin{align}
    \label{map:phi}
        \Phi \colon A^K \to \Fix (H)
    \end{align}
    given by $\Phi(x)=\tilde{x}$ where $\tilde{x}\in A^G$ is the lifting configuration defined by $\tilde{x}(g)= x(\varphi(g))$ for all $g \in G$. Note that $\tilde{x} \in \Fix(H)$ since for all $h \in H$ and $g \in G$, we have 
    $$h\tilde{x}(g)= \tilde{x}(h^{-1}g)=x(\varphi(h^{-1}g)= x(\varphi(h^{-1})\varphi(g))= x(\varphi(g))=\tilde{x}(g).$$
    \par 
    \noindent 
    The inverse of $\Phi$ is the following map 
    \begin{align}
     \label{map:psi}
        \Psi\colon \Fix(H)\to A^K
        \end{align}
        defined by  $\Psi(x)(k)=x(g)$ for all $x\in \Fix(H)$, $k \in K$ and $g\in \varphi^{-1}(k)$. Recall that the homomorphism $\varphi$ is surjective, $\varphi^{-1}(k)= Hg$, and thus $\Psi$ is well defined:  $x(hg) = x(g)$ for all $h \in H$ since $x\in \Fix(H)$.  
    \vspace{0.1cm}
    \par 
    \noindent
    \textbf{The induced cellular automaton $\tau_K$.}
    We can now define $$\tau_K(x)=\Psi(\tau(\Phi(x))) \text{ for all } x\in A^K.
    $$
    Observe that $N=\varphi(M)\subset K$ is a memory set of $\tau_K$. The corresponding local transition map of $\tau_K$ is the map $\mu_K \colon A^N \to A$ given by $\mu_K(y)= \mu(z)$ for all $y \in A^N$, where $z \in A^M$ is defined by $z(m)= y(\varphi(m))$ for all $m \in M$. To summarize, we have a commutative diagram: 
\[
\begin{CD}
\Fix(H) @>{\tau}>> \Fix(H)   \\
@V{\Psi}VV @VV{\Psi}V \\
A^K @>{\tau_K}>> A^K
\end{CD}
\] 
In particular, note that the restriction $\tau\vert_{\Fix(H)}$ is conjugate to the cellular automaton $\tau_K \colon A^K \to A^K$.

\section{The space of marked groups}
\label{s:marked-groups}

 Recall that a group $G$ is a $\Gamma$-marked group if there exists an exact sequence of group homomorphisms 
$$
0 \to H \to \Gamma \to G \to 0. 
$$
We identify $G$ with the normal subgroup $H \subset \Gamma$. Thus, the space of $\Gamma$-marked groups is identified with the space of normal subgroups $\mathcal{N}(\Gamma)$ of $\Gamma$. The space $\mathcal{P}(\Gamma)$ of all subsets of $\Gamma$ is identified with $\{0,1\}^\Gamma$ via the bijection $\chi \colon \PP(\Gamma) \to \{0,1\}^\Gamma$ defined for every $D\in \PP(\Gamma)$ by $\chi(D)(g)=1$ if $g \in D$ and  $\chi(D)(g)=0$ if $g \in \Gamma\setminus D$. 
We equip $\{0,1\}^\Gamma$ with the discrete uniform structure. This induces a uniform structure on the subspace $\mathcal{N}(\Gamma) \subset \mathcal{P}(\Gamma) = \{0,1\}^\Gamma$. Therefore, a base of entourages of $\mathcal{N}(\Gamma) $ is given by the following collection of sets $\{V_E\colon E\subset \Gamma, |E|<\infty\}$ where 
$$
V_E= \{(H_1, H_2)\in \NN(\Gamma)\times \NN(\Gamma)\colon H_1\cap E= H_2\cap E\}. 
$$
These entourages define a base of a  topology on $\NN(\Gamma)$ given by $V_E[H]= \{K\in \NN(\Gamma)\colon (H,K)\in V_E\}$ for all $E\subset \Gamma$ finite. Equipped with this topology, the space $\NN(\Gamma)$ is totally disconnected,  compact, and Hausdorff (see \cite[Proposition~3.4.1]{csc-book}). 
\par 
Let $A$ be an alphabet. Let $W$ be an  entourage  of $A^\Gamma$. For every $D\subset \PP(\Gamma)$, we define the $W$-neighborhood of $D$ in $A^\Gamma$ by 
$$
W[D]= \{x\in A^\Gamma\colon (x,y)\in W \text{ for some } y\in D \}. 
$$
The entourage $W$ induces the following entourage $\widehat{W}$ in the Hausdorff-Bourbaki uniform structure on $\PP(A^\Gamma)$: 
$$
\widehat{W} = \{(X,Y)\in \PP(A^\Gamma)\times \PP(A^\Gamma)\colon X\subset W[Y] \text{ and } Y\subset W[X]\}. 
$$

\section{Post-injunctive groups}
\label{s:post-injunctive-groups}
In some of our proofs, it will be more convenient to work with the following characterizations of post-surjectivity. 
\begin{lemma}
\label{l:post-surj}
 Let $\tau\colon A^G \to A^G$ be a cellular automaton over an alphabet $A$ and a group universe $G$. Then the following conditions are equivalent: 
 \begin{enumerate}[\rm (i)]
 \item $\tau $ is post-surjective; 
\item 
for all $g \in G$ and $x,y \in A^G$ with $y\vert_{G\setminus\{g\}} = \tau(x)\vert_{G\setminus \{g\}}$, there exists $z\in A^G$ asymptotic to $x$ such that $\tau(z)=y$; 
\item 
for all $x,y \in A^G$ with  $y\vert_{G\setminus\{1_G\}} = \tau(x)\vert_{G\setminus \{1_G\}}$, there exists $z\in A^G$ asymptotic to $x$ such that $\tau(z)=y$.
\end{enumerate}
\end{lemma}

\begin{proof}
It follows directly from the definition of post-surjectivity that we have (i)$\implies$ (ii)$\implies$(iii).  We will show that (iii)$\implies$(ii)$\implies$(i). 
\par 
For the implication (iii)$\implies$(ii), assume that $\tau$ satisfies (iii) and consider $g \in G$, $x,y \in A^G$ with $y\vert_{G\setminus\{g\}} = \tau(x)\vert_{G\setminus \{g\}}$. 
By the $G$-equivariance of $\tau$, we have $\tau(g^{-1}x)= g^{-1}\tau(x)$. Moreover,  $(g^{-1}y)\vert_{G\setminus\{1_G\}}= (g^{-1} \tau(x))\vert_{G\setminus\{1_G\}}$ since $y\vert_{G\setminus\{g\}} = \tau(x)\vert_{G\setminus \{g\}}$. It follows that $(g^{-1}y)\vert_{G\setminus\{1_G\}}= \tau(g^{-1}x)\vert_{G\setminus\{1_G\}}$. Applying (iii) to $g^{-1}y$ and $g^{-1}x$, we deduce that there exists $t\in A^G$ asymptotic to $g^{-1}x$ such that $\tau(t)= g^{-1}y$. Consequently, $z=gt\in A^G$ is asymptotic to $x$ and $\tau(z)=  \tau(gt)=g\tau(t)=gg^{-1}y=y$ as desired and (ii) is proved. 
\par 
Assume now that $\tau$ satisfies (ii). Let $x,y \in A^G$ such that $y$ is asymptotic to $\tau(x)$. Then we can find a finite subset $E\subset G$ such that $y\vert_{G\setminus E}= \tau(x)\vert_{G\setminus E}$. We write $E=\{g_1, ..., g_n\}$ and define $z_0=x$. By applying successively $n$ times the property (ii), we obtain a sequence $z_1, ..., z_n \in A^G$ such that for all $k=1,2,...,n$, the configuration $z_k$ is asymptotic to $z_{k-1}$ and $$
\begin{cases}
\tau(z_k)\vert_{G\setminus \{g_k\}} & = \tau(z_{k-1})\vert_{G\setminus \{g_k\}} \\   
\tau(z_k)(g_k) &= y(g_k). 
\end{cases}
$$ 
By the transitivity of the asymptotic relation of configurations, it follows that $z_n$ is asymptotic to $z_0=x$. An immediate induction on $k$ shows that $\tau(z_k)\vert_{\{g_1,...,g_k\}}=y\vert_{\{g_1,...,g_k\}}$ and $\tau(z_k)\vert_{G\setminus E }=\tau(x)\vert_{G\setminus E }$. Consequently, $\tau(z_n)\vert_{E}=y\vert_{E}$ and  $\tau(z_k)\vert_{G\setminus E }=\tau(x)\vert_{G\setminus E }$. Thus $\tau(z_n)=y$ and it follows that $\tau$ is post-surjective. 
We conclude that (ii)$\implies$(i) and the proof is complete.   \end{proof}

The next lemma relates post-surjectivity and pre-injectivity of a cellular automaton and its induced cellular automata over subgroups of the group universe. 
\begin{lemma}
\label{l:pre-post-induction}
    Let $H$ be a subgroup of a group $G$. Let $\tau\colon A^H \to A^H$ be a cellular automaton and let $\sigma \colon A^G \to A^G$ be the induced cellular automata with the same local transition map as $\tau$. The following properties  hold. 
    \begin{enumerate}[\rm (i)]
     \item $\tau$ is post-surjective if and only if $\sigma$ is post-surjective. 
        \item If $\sigma$ is pre-injective then $\tau$ is pre-injective. 
    \end{enumerate}
\end{lemma}

\begin{proof}
Let $M\subset H$ be a finite memory set of both $\tau$ and $\sigma$. Let $\mu \colon A^M \to A$ be their common local transition map. In particular, we have 
 \begin{align}
   \label{eq:induction-proof-1}
   \sigma(x)(g) = \mu((g^{-1}x)\vert_M) \quad \text{for all } x \in A^G, g\in G. 
   \end{align} 
\par 
Let $\Gamma \subset G$ be a complete set of representatives of the quotient $G/H$ such that $1_G\in \Gamma$, that is, $G/ H = \{ \gamma H \colon \gamma \in G \}$.  For each $\gamma\in \Gamma$, consider the following map $\tau_\gamma \colon A^{ \gamma H}\to A^{\gamma H}$ defined for every $x \in A^{\gamma H}$ and $h \in H$ by 
$$
\tau_\gamma(x)( \gamma h ) = \tau(x_\gamma)(h) = \mu((h^{-1}x_\gamma)\vert_M)
$$
where $x_\gamma \in A^H$ is given by $x_\gamma(k)= x(\gamma k)$ for all $k \in H$. Equivalently, 
$$
\tau_\gamma(x)(\gamma h)=\sigma(\tilde{x})(\gamma h) = \mu(((\gamma h)^{-1}\tilde{x})\vert_M)=\mu((h^{-1}x_\gamma)\vert_M)
$$
for every configuration $\tilde{x}\in A^G$ extending $x$. Observe that for every $x= (x\vert_{\gamma H})_{\gamma \in \Gamma} \in \prod_{\gamma \in \Gamma }A^{\gamma H} = A^G$, we have 
$$
\sigma(x) = (\sigma(x)\vert_{\gamma H})_{\gamma \in \Gamma} = (\tau_\gamma(x \vert_{\gamma H}))_{\gamma \in \Gamma}  \in \prod_{\gamma \in \Gamma }A^{\gamma H} = A^G.   
$$
In other words, we have a decomposition 
\begin{align}
\label{e:decomposition-tau-sigma-coset}
\sigma = \prod_{\gamma \in \Gamma} \tau_\gamma \,  \colon\,  \prod_{\gamma \in \Gamma }A^{\gamma H} \to \prod_{\gamma \in \Gamma }A^{\gamma H}. 
\end{align}
\par 
For (i), suppose first that $\tau$ is post-surjective. 
To show that $\sigma$ is also post-surjective, it suffices to check that $\sigma$ satisfies the condition (iii) in Lemma~\ref{l:post-surj}. For this, let $x,y \in A^G$ be such that $y\vert_{G\setminus \{1_G\}}=\sigma(x)\vert_{G\setminus \{1_G\}}$. Let $u=x\vert_H$ and  $v=y\vert_H$. Since the memory set $M$ of both $\sigma$ and $\tau$ is a subset of $H$, it follows from \eqref{eq:induction-proof-1} that $v\vert_{H\setminus \{1_H\}}=\tau(u)\vert_{H\setminus \{1_H\}}$. By  the post-surjectivity of $\tau$, there exists a configuration $w\in A^H$ asymptotic to $u$  such that $\tau(w)=v$. Let $z\in A^G$ be the configuration defined by $z\vert_H= w$ and $z\vert_{G\setminus H} = x\vert_{G\setminus H}$. Since $z\vert_H=w$ is asymptotic to $u=x\vert_H$ and $z\vert_{G\setminus H}=x\vert_{G\setminus H}$, the configurations $z$ and $x$ are asymptotic. We claim  that $\sigma(z)=y$.  Indeed,  for every $g \in G\setminus H$, we have $gM \subset G\setminus H$ (since $H$ is a subgroup of $G$) and thus $(g^{-1}z)\vert_M = (g^{-1}x)\vert_M$. In particular, it holds for all $g \in G\setminus H$ that 
   $$\sigma(z)(g)=\mu((g^{-1}z)\vert_M)= \mu((g^{-1}x)\vert_M)=\sigma(x)(g).$$
   Hence, $\sigma(z)\vert_{G\setminus H} = \sigma(x)\vert_{G\setminus H}$. Since $y\vert_{G\setminus \{1_G\}}=\sigma(x)\vert_{G\setminus \{1_G\}}$, we obtain 
   \begin{align}
   \label{eq:proof-subgrou-1}
\sigma(z)\vert_{G\setminus H} = y\vert_{G\setminus H}. 
   \end{align}
\noindent 
   Let $g \in H$. Then $gM \subset H$ and  $(g^{-1}z)\vert_M = (g^{-1}z\vert_H)\vert_M$. Consequently, $$ 
   \sigma(z)(g)=\mu((g^{-1}z)\vert_M)=  \mu((g^{-1}z\vert_H)\vert_M) = \tau(z\vert_H)= \tau(w)=v = y \vert_H.$$
   Therefore, 
    \begin{align}
   \label{eq:proof-subgrou-2}
\sigma(z)\vert_{H} = y\vert_{H}. 
   \end{align} 
   We infer from \eqref{eq:proof-subgrou-1} and \eqref{eq:proof-subgrou-2} that $\sigma(z)=y$ as claimed. It follows that $\sigma\colon A^G\to A^G$ is post-surjective by Lemma~\ref{l:post-surj}. In other words, if $\tau$ is post-surjective then so is $\sigma$. 
   \par 
Conversely, suppose that $\sigma$ is post-surjective. Let $x,y \in A^H$ be such that $y\vert_{H \setminus \{1_H\}} = \tau(x)\vert_{H \setminus \{1_H\}}$. Let $c\in A^{G}$ be an arbitrary configuration. Let $\tilde{x}, \tilde{y}\in A^G$ be the extensions of $x,y$ given by $\tilde{x}\vert_H = x$,  $\tilde{x}\vert_{G\setminus H}=c\vert_{G\setminus H}$, and $\tilde{y}\vert_H = y$, $  \tilde{y}\vert_{G\setminus H} = \sigma(c)\vert_{G\setminus H}$.  
Since $\tilde{x}\vert_{G\setminus H}=c\vert_{G\setminus H}$, we deduce from  the decomposition 
\eqref{e:decomposition-tau-sigma-coset}  that $ \sigma(\tilde{x})\vert_{G\setminus H}=\sigma(c)\vert_{G\setminus H}=\tilde{y}\vert_{G\setminus H} $. 
Moreover,  
\begin{align*}
\sigma(\tilde{x})\vert_H  = \tau (\tilde{x}\vert_{H})  
 = \tau(x). 
\end{align*}
Since $\tau(x)_{H \setminus \{1_H\}}=y_{H \setminus \{1_H\}}$, we conclude that 
$$
\sigma(\tilde{x})_{G \setminus \{1_G\}} = y_{G \setminus \{1_G\}}. 
$$
Since $\sigma$ is post-surjective, there exists $z\in A^G$ asymptotic to $\tilde{x}$ such that $\sigma(z)= \tilde{y}$. In particular, $z\vert_{H}$ is asymptotic to $\tilde{x}\vert_H= x$. Since $\sigma(z)= \tilde{y}$,  the decomposition 
\eqref{e:decomposition-tau-sigma-coset} implies that 
$$
\tau(z\vert_H) = \tilde{y}\vert_H =y. 
$$
This proves that $\tau$ is also post-surjective which completes the proof of (i).  
   \par
  For (ii), suppose first that $\tau$ is pre-injective. We need to prove that $\sigma$ is also pre-injective. Suppose on the contrary that there exist distinct asymptotic configurations $x,y \in A^G$ such that  $\sigma(x)=\sigma(y)$. By the decomposition 
\eqref{e:decomposition-tau-sigma-coset}, it follows that for every $\gamma \in \Gamma$, we have $\tau_\gamma(x\vert_{\gamma H}) = \tau_\gamma(y\vert_{\gamma H})
$ or equivalently, 
\begin{align}
\label{e:induction-post-sur-proof-3}
\tau(x_\gamma) = \tau (y_\gamma). 
\end{align} 
Since $x=(x\vert_{\gamma H})_{\gamma \in \Gamma}$ and $y=(y\vert_{\gamma H})_{\gamma \in \Gamma}$ are distinct, there must exist $\alpha \in \Gamma$ such that $x\vert_{\alpha H}\neq y\vert_{\alpha H}$ or equivalently, $x_\alpha \neq y_\alpha$. Since $x$ and $y$ are asymptotic, so are $x_\alpha$ and $y_\alpha$. But $\tau(x_\alpha)=\tau(y_\alpha)$ by \eqref{e:induction-post-sur-proof-3}, we conclude that $\tau$ is not pre-injective. This contradiction shows that $\sigma$ must be pre-injective. 
\par 
Conversely, suppose that $\sigma$ is pre-injective but $\tau$ is not pre-injective. Then there exists distinct  asymptotic configurations $x,y \in A^H$ such that $\tau(x)=\tau(y)$. Let $\tilde{x}, \tilde{y}\in A^G$ be extension of $x,y$ such that $\tilde{x}\vert_{G\setminus H}= \tilde{y}\vert_{G\setminus H}$. Then $\tilde{x}$ and $\tilde{y}$ are also  distinct  asymptotic configurations. Since $\tilde{x}\vert_{H}=x$, $ \tilde{y}\vert_{H}=y$, and $\tilde{x}\vert_{G\setminus H}= \tilde{y}\vert_{G\setminus H}$, the decomposition 
\eqref{e:decomposition-tau-sigma-coset} tells us that 
\begin{align*}
\sigma(\tilde{x})   & =\tau(\tilde{x}\vert_{H})\times \prod_{\gamma \in \Gamma \setminus\{1_G\}}\tau_\gamma(\tilde{x}\vert_{\gamma H}) \\ &  = \tau(x)\times \prod_{\gamma \in \Gamma \setminus\{1_G\}}\tau_\gamma(\tilde{x}\vert_{\gamma H})  \\ 
& = \tau(y)\times \prod_{\gamma \in \Gamma \setminus\{1_G\}}\tau_\gamma(\tilde{y}\vert_{\gamma H})\\
 & =\tau(\tilde{y}\vert_{H})\times \prod_{\gamma \in \Gamma \setminus\{1_G\}}\tau_\gamma(\tilde{y}\vert_{\gamma H})\\
 & = \sigma(y). 
\end{align*}
Hence, $\sigma$ is not pre-injective, contradicting the hypothesis on $\sigma$. Therefore, $\tau$ must be pre-injective and the proof of (ii) is complete. 
\end{proof}

The following uniform post-surjectivity result for cellular automata is a special case of the same property for non-uniform cellular automata with finite memory \cite[Lemma~13.2]{phung-tcs}.  

\begin{theorem}[Uniform post-surjectivity]
\label{t:uniform-post-surj}
Let $\tau \colon A^G\to A^G$ be a post-surjective cellular automaton with finite alphabet $A$ and countable group universe $G$. Then there exists a finite subset $F \subset G$ such that for every $x,y\in G$ and every $g\in G$ such that $y\vert_{G\setminus \{g\}}= \tau(x)\vert_{G\setminus \{g\}}$, there exists a configuration $z\in A^G$ such that $z\vert_{G\setminus (gE)} = x\vert_{G\setminus (gE)}$ and $\tau(z)=y$. 
\end{theorem}

 \section{On fully residually and locally post-injunctive groups}
 \label{s:fully-res-locally-post-injunctive}

Our first observation is that every extension of a post-injunctive group with finite kernel by a post-injunctive group must be post-injunctive. 
\begin{lemma}
\label{l:virtual-post-injunctive}
    Virtually post-injunctive  groups are post-injunctive. 
\end{lemma}
\begin{proof}
    Suppose that $G$ has a finite index post-injunctive subgroup $H$. Let $\tau\colon A^G\to A^G$ be a post-surjective  cellular automaton. 
    Denote $B=A^{G/H}$. Let $S\subset G$ be a complete set of representatives of $G/H$ so that every element $g\in G$ can be written in a unique way as $g=s_g h_g$ where $s_g \in S$ and $h_g \in H$. Let $\pi\colon G\to G/H$ be the projection. We then have an $H$-equivariant bijection $\varphi \colon B^H \to A^G$ given by $\varphi(x)(g)=x(h_g) (\pi(s_g))$ for all $x\in B^H$ and $g \in G$. Moreover, 
    $\tau$ is conjugate to a cellular automaton $\sigma \colon B^H \to B^H$ defined by $\sigma(s)=\varphi^{-1}(\tau(\varphi(x)))$ for all $x \in B^H$. Since $G/H$ is finite, it is clear two configurations $x,y \in A^G$ are distinct, resp.  asymptotic, if and only if so are $\varphi(x), \varphi(y) \in B^H$. Consequently, it follows from the definition of post-surjectivity and pre-injectivity that $\sigma$ is post-surjective, resp. pre-injective, if and only if so is $\tau$.  Since $\tau$ is post-surjective, so is  $\sigma$. As $H$ is post-injective, $\sigma$ must be injective and thus so is $\tau$. This completes the proof  that $G$ is post-injunctive. 
\end{proof}

\begin{corollary}
\label{cor:finite-kernel}
Semidirect extension
with finite kernel of a post-injunctive group is post-injunctive.\qed 
\end{corollary}

Conversely, we show that subgroups of a post-injunctive group must be also post-injunctive.

 \begin{theorem}
 \label{t:subgroup-post-injunctive}
Post-injunctivity is stable by taking subgroups.  More specifically, every subgroup of a post-injunctive group is also  post-injunctive. 
 \end{theorem}

\begin{proof}
    Let $G$ be a post-injunctive group and let $H\subseteq G$ be an arbitrary subgroup. Let $A$ be a finite alphabet and let $\tau\colon A^H \to A^H$ be a post-surjective cellular automata with finite memory set $M \subseteq H$ and local transition map $\mu\colon A^M \to A$.  We must prove that $\tau$ is also pre-injective. 
  Since $M \subset H \subset G$, the cellular automata $\tau$ induces a cellular automaton $\sigma \colon A^G \to A^G$ which admits the same local transition map $\mu$ as $\tau$:  
   \begin{align}
   \sigma(x)(g) = \mu((g^{-1}x)\vert_M), \quad \text{for all } x \in A^G, g\in G. 
   \end{align}
   Since $\tau$ is post-surjective, we infer from  Lemma~\ref{l:pre-post-induction} that $\sigma$ is also post-surjective. As the group $G$ is post-injunctive, $\sigma$ must be pre-injective. Again by Lemma~\ref{l:pre-post-induction}, we deduce that $\tau$ is pre-injective. This proves that $H$ is a post-injunctive group and the proof of the theorem is complete.  
\end{proof}

Let $(P)$ be a property of groups. A group $G$ is said to be 
\begin{enumerate}[\rm (i)]
    \item locally $(P)$ if every finitely generated subgroup of $G$ satisfies $(P)$; 
    \item fully residually $(P)$ if for every finite subset $E \subset G$, there exist a group $H$ satisfying $(P)$ and a surjective group homomorphism $\varphi \colon G \to H$ such that $\varphi\vert_E$ is injective. 
\end{enumerate}

\begin{theorem}
\label{t:locally-post-injunctive}
    A group is post-injunctive if and only if it is locally post-injunctive. 
\end{theorem}

\begin{proof}
Let $G$ be a group. If $G$ is post-injunctive then so is every subgroup of $G$ by Theorem~\ref{t:subgroup-post-injunctive} and thus every finitely generated subgroup of $G$ is post-injunctive. Conversely, suppose that $G$ is locally post-injunctive. Let $\tau \colon A^G \to A^G$ be a post-surjective cellular automaton. Let $M\subset G$ be a finite memory set of $\tau$ and let $\mu \colon A^M \to A$ be the corresponding local transition map. Let $H \subset G$ be the finitely generated subgroup generated by $M$. Since $M\subset H$, the restriction of $\tau$ to $A^H$ defines a cellular automaton $\sigma = \tau\vert_{A^H}\colon A^H \to A^H$ with the same memory set $M$ and local transition map $\mu$ as $\tau$. Since $\tau$ is post-surjective, Lemma~\ref{l:pre-post-induction} implies that $\sigma$ is also post-surjective. As $G$ is locally post-injunctive, the subgroup $H$ is post-injunctive. It follows that $\sigma$ is pre-injective. Again by Lemma~\ref{l:pre-post-induction}, we conclude that $\tau$ is also pre-injective. This proves that $G$ is a post-injunctive group.  
\end{proof}

Our next goal of this section is to prove the following result concerning fully residually post-injunctive groups. 

\begin{theorem}
\label{t:fully-residually-post-injunctive}
    Fully residually post-injunctive groups are post-injunctive.   
\end{theorem}

\begin{proof}
    Let $A$ be a finite alphabet and let $G$ be a fully residually post-injunctive group. Let  $\tau \colon A^G \to A^G$ be a post-surjective cellular automaton. We need to show that $\tau$ is pre-injective. Suppose on the contrary that there exist two distinct asymptotic configurations $c, d \in A^G$ such that $\tau(c)=\tau(d)$. Then for some finite subset $E\subset G$, we have $c\vert_{G\setminus E}= d\vert_{G\setminus E}$ and  $c\vert_E\neq d\vert_E$. 
    \vspace{0.15cm}
    \par 
    \noindent 
    \textbf{Step 1: recall the auxiliary maps $\Phi$, $\Psi$, and $\tau_K$ (Section~\ref{s:map-phi-psi})}. 
    Let $M\subset G$ be a finite memory set of $\tau$ and let $\mu \colon A^M \to A$ be the associated local transition map. The projection map $\varphi \colon G\to K$ induces    a canonical $K$-equivariant bijection  
    $\Phi \colon A^K \to \Fix (H)$  (see \eqref{map:phi})
   with  inverse  $\Psi\colon \Fix(H)\to A^K$ defined by  $\Psi(x)(k)=x(g)$ for all $x\in \Fix(H)$, $k \in K$ and $g\in \varphi^{-1}(k)$ (see \eqref{map:psi}).  
Then $\tau$  induces the quotient cellular automaton $\tau_K \colon A^K \to A^K$ defined by $\tau_K(x)=\Psi(\tau(\Phi(x)))$ for all $x\in A^K$. The finite set $N=\varphi(M)\subset K$ is a memory set of $\tau_K$ with corresponding local transition map  $\mu_K \colon A^N \to A$ given by $\mu_K(y)= \mu(z)$ for all $y \in A^N$, where $z \in A^M$ is defined by $z(m)= y(\varphi(m))$ for all $m \in M$.
The following diagram is commutative: 
\[
\begin{CD}
\Fix(H) @>{\tau}>> \Fix(H)   \\
@V{\Psi}VV @VV{\Psi}V \\
A^K @>{\tau_K}>> A^K
\end{CD}
\] 
     \par 
Returning to our proof, Theorem~\ref{t:uniform-post-surj} and the  post-surjectivity of $\tau$  imply that there exists a finite subset $F \subset G$ such that for every $x,y\in G$ and every $g\in G$ such that $y\vert_{G\setminus \{g\}}= \tau(x)\vert_{G\setminus \{g\}}$, there exists a configuration $z\in A^G$ such that $z\vert_{G\setminus (gE)} = x\vert_{G\setminus (gE)}$ and $\tau(z)=y$. Up to enlarging $F$, we can assume without loss of generality that $M\subset F$, $1_G\in F$, and $F$ is symmetric, that is, $F=F^{-1}$.   
\par 
Let $\Omega=EF^2\cup F^3\supset F^2\supset F \supset M$. Since $G$ is fully residually post-injunctive, there exists a surjective homomorphism of groups $\varphi \colon G \to K$ such that $K$ is post-injunctive and $\varphi\vert_\Omega$ is injective. 
\vspace{0.15cm}
\par 
\noindent
\textbf{Step 2: proving the post-surjectivity of $\tau_K \colon A^K \to A^K$.} 
Let $x,y \in A^K$ be such that $y\vert_{K \setminus \{1_K\}} = \tau_K(x)\vert_{K \setminus \{1_K\}}$. Let $u=\Phi(\tau_K(x)) = \tau(\Phi(x))\in \Fix(H)$ and $v= \Phi(y) \in \Fix(H)$. We define $w \in A^G$ by setting $w\vert_{G\setminus \{1_G\}}= u\vert_{G\setminus \{1_G\}}$ and $w(1_G)= v(1_G)$. Then  $w\vert_{G\setminus H} = u\vert_{G\setminus H}=  v\vert_{G\setminus H}$. By the choice of $F$ and the post-surjectivity of $\tau$ applied to $\Phi(x)$ and $w$, there exists $t\in A^G$ such that $t\vert_{G\setminus F} = \Phi(x)\vert_{G\setminus F}$ and $\tau(t)=w$. Consider the following configuration $s\in A^G$ defined by 
$$
(*)\,\begin{cases}
    s(hg) = t(g) \, \text{ for all } h \in H,  g \in F.\\ 
    s\vert_{G\setminus (HF)}  = t\vert_{G\setminus (HF)}.
\end{cases}
$$
Note that $s\vert_{HF}$ is well-defined for if $h_1g_1=h_2g_2$ where $h_1,h_2\in H$ and $g_1,g_2\in F$ then as  $H=\ker \varphi$, we have 
$$\varphi(g_1)=\varphi(h_1g_1)=\varphi(h_2g_2)=\varphi(g_2)$$
and thus $g_1=g_2$ since $\varphi\vert_F$ is injective. 
\vspace{0.15cm}
\par 
\noindent 
\textbf{Claim: $s$ is $H$-periodic.} To see this, let $h \in H$ and $g \in G$. If $g\in HF$ then  $g=h' g'$ for some $h'\in H$ and $g'\in F$  and by $(*)$, we have  $$hs(g)=s(h^{-1}g)= s(h^{-1}h'g')=t(g')=s(h'g')=s(g)$$  because $h^{-1}h', h' \in H$. If $g\in G\setminus (HF)$ then $ h^{-1}g\in G\setminus(HF)$ and since $ s\vert_{G\setminus (HF)}  = t\vert_{G\setminus (HF)}$, $t\vert_{G\setminus F}= \Phi(x)\vert_{G\setminus F}$, $F\subset HF$, and $\Phi(x)\in \Fix(H)$, it follows that  
$$hs(g)= s(h^{-1}g)= t(h^{-1}g)= \Phi(x)(h^{-1}g)=\Phi(x)(g)= t(g)=s(g).$$  
Therefore,  $hs(g)= s(g)$ for all $h \in H$ and all $g \in G$. We deduce that  $s\in \Fix(H)$ as claimed. 
\par 
Define $z= \Psi(s)\in A^K$. For all $k \in K \setminus \varphi(F)$ and $g\in \varphi^{-1}(k)$ so that $g\in G\setminus (HF)$ (as $H=\ker \varphi$), we have  
\begin{align*}
z(k)= \Psi(s)(k)& = s(g) & \text{ (by definition of } z \text{ and } \Psi)  \\
& = t(g) & \text{ (by }  (*)) \\
& =\Phi(x)(g) & \text{ (as } t\vert_{G\setminus F} = \Phi(x)\vert_{G\setminus F})  \\
& = x(\varphi(g)) & \text{ (by definition of } z \text{ and } \Phi)\\
& =x(k) &  \text{ (as } g\in \varphi^{-1}(k)).
\end{align*}
\noindent 
This proves that $z\vert_{K \setminus \varphi(F)} = x\vert_{K \setminus \varphi(F)}$, that is, $z$ and $x$ are  asymptotic. 
\vspace{0.15cm}
\par 
\noindent 
\textbf{Claim: $\tau_K(z)=y$.} For this, 
first note that 
$\tau_K(z) = \tau_K(\Psi(s))= \Psi(\tau(s))$. 
Let $k \in K \setminus \varphi(F^2)$ and let $g \in \varphi^{-1}(k)$. Then $g\notin HF^2$ and thus $gF\cap HF = \varnothing$ (note that $F=F^{-1}$). 
Hence, 
$gM \subset gF \subset  G\setminus (HF)$. By $(*)$, we obtain $(g^{-1}s)\vert_M = (g^{-1}t)\vert_M$. Since $\tau(t)=w$, it follows from  the choice of $w, u$  that 
\begin{align*}
\tau_K(z)(k)& = \Psi(\tau(s))(k) = \tau(s)(g) = \mu((g^{-1}s)\vert_M) =\mu((g^{-1}t)\vert_M) \\ &  = \tau(t)(g) =w(g)=u(g)= \Phi(\tau_K(x))(g) = \tau_K(x)(k).  
\end{align*}
In other words, we have 
\begin{align}
\label{e:proof-fully-residually-1}
\tau_K(z)\vert_{K \setminus \varphi(F^2)}= \tau_K(x)\vert_{K \setminus \varphi(F^2)}. 
\end{align}
\par 
Now let $k \in \varphi(F^2)$. Then $k= \varphi(g)$ for some $g \in F^2$. Since $\varphi\vert_{F^3}$ is injective and $gM \subset  F^2M\subset F^3$, we have $(g^{-1}s)\vert_M = (g^{-1}t)\vert_M$. Combining with the relation $\tau(t)=w$, we find that 
\begin{align}
\label{e:proof-fully-residually-2}
\tau_K(z)(k)& = \Psi(\tau(s))(k) = \tau(s)(g) \nonumber \\& = \mu((g^{-1}s)\vert_M) =\mu((g^{-1}t)\vert_M) \nonumber \\ &  = \tau(t)(g) =w(g).  
\end{align}
If $k=1_K$ then $g=1_G$ by the injectivity of $\varphi\vert_{F^2}$ as $g\in F^2\cap \varphi^{-1}(k)$ and we deduce from \eqref{e:proof-fully-residually-2} that 
\begin{align}
    \label{e:proof-fully-residually-3} \tau_K(z)(1_K)=w(1_G) = v(1_G) = \Phi(y)(1_G)=y(1_K).
\end{align}
If $k \in \varphi(F^2)\setminus\{1_K\}$ then $k= \varphi(g)$ for some $g \in F^2\setminus\{1_G\}$. Hence, the relation~\eqref{e:proof-fully-residually-2} and the choice of $w,u$ imply that 
\begin{align*}
\tau_K(z)(k)=w(g) = u(g) = \Phi(\tau_K(x))(g)= \tau_K(x)(k). 
\end{align*}
In other words, we have 
\begin{align}
    \label{e:proof-fully-residually-4} 
    \tau_K(z)\vert_{\varphi(F^2)\setminus\{1_K\}}=\tau_K(x)\vert_{\varphi(F^2)\setminus\{1_K\}}. 
\end{align}
The relations  \eqref{e:proof-fully-residually-1},  \eqref{e:proof-fully-residually-3}, and  \eqref{e:proof-fully-residually-4} prove the claim that $\tau_K(z)=y$. We deduce from Lemma~\ref{l:post-surj}.(ii) that $\tau_K$ is indeed post-surjective. 
\vspace{0.15cm}
\par
\noindent
\textbf{Step 3: Obtaining a contradiction.}  
Let $a\in A$ be a fixed element and consider the following configurations $p,q \in A^K$ defined by 
$$
\begin{cases}
    p(\varphi(g))= c(g) & \text{ for all } g \in EF^2,\\
    q(\varphi(g))= d(g) & \text{ for all } g \in EF^2,\\
    p(k) = q(k)=a & \text{ for all } k \in K \setminus \varphi(EF^2).\\
\end{cases}
$$
Recall that $c\vert_{G\setminus E}= d\vert_{G\setminus E}$ and  $c\vert_E\neq d\vert_E$. 
Since $E\subset EF^2 \subset \Omega$ and $\varphi\vert_\Omega$ is injective, we deduce that $p,q\in A^K$ are well-defined and are distinct asymptotic configurations with  $p\vert_{K\setminus \varphi(E)}= d\vert_{G\setminus \varphi(E)}$. 
\par 
We claim that $\tau_K(p)=\tau_K(q)$. Indeed, if $k \in K \setminus \varphi(EF)$ then $kM\subset K \setminus \varphi(E)$ so  $(k^{-1}p)\vert_{N} = (k^{-1}q)\vert_{N}$ and thus $\tau_K(p)(k)=\tau_K(q)(k)$ (recall that $N=\varphi(M)$ is a memory set of $\tau_K$). If $k= \varphi(g)\in \varphi(EF)$ for some $g \in EF$ then $gM\subset EFM\subset EF^2$ and we have by the construction of $p,q$ that  
\begin{align*}
\tau_K(p)(k)& =\mu_K((k^{-1}p)\vert_{N})=\mu((g^{-1}c)\vert_{M}) = \tau(c)(g)\\
& =\tau(d)(g) = \mu((g^{-1}d)\vert_{M}) =\mu_K((k^{-1}q)\vert_{N}) =  \tau_K(q)(k). 
\end{align*}
It follows that $\tau_K(p)(k)= \tau_K(q)(k)$ for all $k \in K$, that is, $\tau_K(p)=\tau_K(q)$. Therefore, $\tau_K$ is not pre-injective, which is a contradiction since $K$ is a post-injunctive group and $\tau_K\colon A^K \to A^K$ is post-surjective by Step 2. 
\vspace{0.15cm}
\par 
\noindent
\textbf{Step 4: conclusion.} It follows from the above that every post-surjective cellular automaton $\tau\colon A^G \to A^G$ must be pre-injective. We conclude that $G$ is post-injunctive and the proof of the theorem is complete.  
\end{proof}

  \section{Closedness of post-injunctive groups under certain semidirect extensions}
\label{s:mark-groups-proof}
As a consequence of Lemma~\ref{l:virtual-post-injunctive} and Theorem~\ref{t:fully-residually-post-injunctive}, we obtain the following stable properties under certain semidirect extension of post-injunctive groups. 

\begin{theorem}
\label{t:semdirect-extension-res-finite-kernel}
Every semidirect extension of a post-injunctive group with  finitely
generated residually finite kernel is also a post-injunctive group. More specifically, if a group $G$ satisfies a semidirect  extension of groups 
$$
0 \to K \to G \to H\to 0.
$$
where  $H$ is post-injunctive  and  $K$ is  finitely generated residually finite, then    $G$ is also a post-injunctive group.   
\end{theorem}

\begin{proof}
By Corollary~\ref{cor:finite-kernel} of Lemma~\ref{l:virtual-post-injunctive}, every semidirect extension with finite kernel of a post-injunctive group is  also a post-injunctive group. By Theorem~\ref{t:fully-residually-post-injunctive}, fully residually post-injunctive groups are also post-injunctive. Therefore, the split case of 
    \cite[Theorem~1]{arzhantseva} implies that every semidirect extension of a post-injunctive groups with a finitely
generated residually finite kernel is post-injective.  
\end{proof}

\section{Around Gromov's injectivity lemma}
\label{s:analogues-gromov-lemma}

Gromov's injectivity lemma (cf. \cite[Lemma 4.H"]{gromov-esav} and \cite[Theorem~3.6.1]{csc-book}) states that every injective cellular automaton $\tau \colon \Sigma \to \Sigma$ over a subshift $\Sigma\subset A^G$ can be extended to an injective cellular automaton over every subshift $X\subset A^G$ which lies in a nontrivial open neighborhood of  $\Sigma$. 
\par 
In this section, we prove some  analogues of Gromov's injectivity lemma (see Lemma~\ref{l:pre-injective-mark-group} and Lemma~\ref{l:post-surjective-mark-group}) where injectivity is replaced by other fundamental properties such as pre-injectivity and post-surjectivity. Some counterexamples are also given.  
\par 
We begin with the following lemma which shows that the limit of a certain sequence of pre-injective, resp. surjective, cellular automata is also a pre-injective, resp. surjective, cellular automaton. 

\begin{lemma}
    \label{l:pre-injective-mark-group} Let $(H_i)_{i \in I}$ be a net in $\NN(\Gamma)$ of normal subgroups of a group $\Gamma$ which converges to some $H \in \NN(\Gamma)$. Let  $\widetilde{\tau}\colon A^\Gamma \to A^\Gamma$ be a cellular automaton. Let $\tau_i \colon A^{\Gamma/H_i} \to A^{\Gamma/H_i}$ for $i\in I$ and let $\tau \colon A^{\Gamma/H} \to A^{\Gamma/H}$  be the  cellular automata induced by $\widetilde{\tau}$. 
    Then the following properties hold. 
    \begin{enumerate}[\rm (i)]
        \item If  $\tau_i$ is pre-injective for every $i \in I$ then $\tau$ is also pre-injective; 
         \item If  $\tau_i$ is surjective for every $i \in I$ then $\tau$ is also surjective.  
    \end{enumerate}

\end{lemma}

\begin{proof}
For (i), assume  that $\tau_i$ is pre-injective for every $i \in I$.  
    Suppose on the contrary that $\tau$ is not pre-injective. This means that there exist distinct configurations $x,y \in Z=\Fix(H)$ and a finite subset $F\subset \Gamma$ such that $\widetilde{\tau}(x)=\widetilde{\tau}(y)$ and $x\vert_{\Gamma \setminus (HF)}= y\vert_{\Gamma \setminus (HF)}$ but $x\vert_{F} \neq y\vert_{F}$. Up to enlarging $F$, we can clearly suppose that $F=F^{-1}$.  \par 
Let $Z=\Fix(H)$ and $Z_i=\Fix(H_i)$ for $i \in I$. Note that the restriction $\widetilde{\tau}\vert_{Z}\colon Z \to Z$ is conjugate to $\tau $ and similarly, the restriction $\widetilde{\tau}\vert_{Z_i}\colon Z_i \to Z_i$ is conjugate to  $\tau_i$ for every $i \in I$.
    \par 
Let $\pi\colon \Gamma \to \Gamma/H$ and $\pi_i \colon \Gamma \to \Gamma/H_i$ for $i \in I$ be the projections. 
Let $\widetilde{M}\subset \Gamma$ be a finite memory set of $\widetilde{\tau}$ and let $\widetilde{\mu} \colon A^M \to A$ be the corresponding local transition map. Without loss of generality, assume that  $1_G\in \widetilde{M}$ and that $\widetilde{M}$ is symmetric, i.e., $\widetilde{M}=(\widetilde{M})^{-1}$. Consider the following finite subset  
$$
E= F (\widetilde{M})^2 \supset F \widetilde{M}\supset F.  
$$
\par 
Let $W= \{(x,y)\in A^\Gamma \times A^\Gamma\colon x\vert_{E} = y\vert_{E} \}$ be an entourage of $A^\Gamma$. By the result \cite[Theorem~3.4.4]{csc-book}, 
 the convergence of the net $(H_i)_{i \in I}$ to $H$ in $\NN(\Gamma)$ implies the convergence of the net $(Z_i)_{i \in I}$ to $Z$ in the Hausdorff-Bourbaki topology on $\PP(A^\Gamma)$. 
We thus obtain some $j\in I$ such that  $(H,H_j)\in V_{E^2}$ and $(Z,Z_j)\in \widehat{W}$,   
Therefore, $H\cap E^2= H_j\cap E^2$ and $Z_j\subset W[Z]$. Thus there exist
  configurations $u,v \in Z_j$ with  $(u,x), (v,y)\in W$. In particular, we have 
\begin{align}
\label{lemma:pre-injective-quotient-limit}
    u\vert_{E} =  x\vert_{E}, \qquad   v\vert_{E} =   y\vert_{E}. 
\end{align}
We claim that $E \cap HF= E\cap H_j F$. Indeed, let $f \in F$. Suppose that $hf=g\in E$. If $h \in H$ then $h =gf^{-1}\in EF^{-1}=EF\subset E^2$ and thus $h \in E^2\cap H=E^2\cap H_j \subset H_j$. Therefore, $E \cap HF \subset E \cap H_jF$. A similar argument shows the converse inclusion and we conclude that 
\begin{align}
\label{lemma:pre-injective-quotient-limit-5}
   E \cap HF = E \cap H_jF. 
\end{align}
\par 
We fix an arbitrary constant configuration   $c\in \Fix(\Gamma)$. Let us define  $z,t\in Z_j$ by setting 
\begin{align}
\label{lemma:pre-injective-quotient-limit-2}
\begin{cases}
   z\vert_{H_jE} = u\vert_{H_jE},
    \\
    t\vert_{H_jE} = v\vert_{H_jE}, \\
    z\vert_{\Gamma \setminus (H_jE)} =  t\vert_{\Gamma \setminus (H_jE)}=c\vert_{\Gamma \setminus (H_jE)}. 
\end{cases}
\end{align}
Note that $z,t\in \Fix(H_j)=Z_j$ since $u,v,c\in \Fix(H_j)$. Since $E$ is finite, it follows from \eqref{lemma:pre-injective-quotient-limit-2} that the image configurations $\Psi_j(z), \Psi_j(t)\in A^{\Gamma/H_j}$ under the canonical bijection $\Psi_j \colon \Fix(H_j) \to A^{\Gamma/H_j}$ (see Section~\ref{s:map-phi-psi}) are asymptotic. Moreover, $\Psi_j(z)$ and $ \Psi_j(t)$ are distinct since $F\subset E\subset H_jE$ and by \eqref{lemma:pre-injective-quotient-limit}, 
$$
z\vert_F = u\vert_F = x\vert_F  \neq y\vert_F= v\vert_F=t\vert_F,  
$$
and thus for any $g \in F$ such that $z(g) \neq t(g)$, we have 
$$\Psi_j(z)(\pi_j(g))= z(g)\neq t(g)= \Psi_j(t)(\pi_j(g)).$$
\par 
\textbf{Claim:} $\widetilde{\tau}(z) = \widetilde{\tau}(t)$. Indeed, let $hg \in H_jF\widetilde{M}$ where $h\in H_j$ and $g \in F\widetilde{M}$. Then $g\widetilde{M}\subset F(\widetilde{M})^2=E$ and in particular $hg\in H_jE$.  Combining with the relations   \eqref{lemma:pre-injective-quotient-limit} and \eqref{lemma:pre-injective-quotient-limit-2}, we find that 
\begin{align} 
\label{lemma:pre-injective-quotient-limit-1}
(g^{-1}z)\vert_{\widetilde{M}}&=(g^{-1}u)\vert_{\widetilde{M}}= (g^{-1}x)\vert_{\widetilde{M}},\\ (g^{-1}t)\vert_{\widetilde{M}}&=(g^{-1}v)\vert_{\widetilde{M}}=(g^{-1}y)\vert_{\widetilde{M}}.\nonumber \end{align}
 Using  the hypothesis $\widetilde{\tau}(x)=\widetilde{\tau}(y)$,  the relation \eqref{lemma:pre-injective-quotient-limit-1}, and the fact that $\widetilde{\tau}(z), \widetilde{\tau}(t)\in \Fix(H_j)$ since $z,t\in \Fix(H_j)$, we can thus compute 
\begin{align*}
    \widetilde{\tau}(z)(hg)    = \widetilde{\tau}(z)(g)&=   \widetilde{\mu}((g^{-1}z)\vert_{\widetilde{M}}) = \widetilde{\mu}((g^{-1}x)\vert_{\widetilde{M}})\\
    & = \widetilde{\tau}(x)(g) =  \widetilde{\tau}(y)(g) \\& =  \widetilde{\mu}((g^{-1}y)\vert_{\widetilde{M}})= \widetilde{\mu}((g^{-1}t)\vert_{\widetilde{M}}) = \widetilde{\tau}(t)(g) = \widetilde{\tau}(t)(hg). 
\end{align*}
This proves that 
\begin{align}
\label{e:lemma-pre-1}
\widetilde{\tau}(z)\vert_{H_jF\widetilde{M}}= \widetilde{\tau}(t)\vert_{H_jF\widetilde{M}}. 
\end{align}
Now suppose that $g\in \Gamma \setminus (H_j F\widetilde{M})$. Since $(\widetilde{M})^{-1}=\widetilde{M}$, we deduce that $
g\widetilde{M} \subset \Gamma \setminus (H_j F)$. 
We will show that $z\vert_{g\widetilde{M}}= t\vert_{g\widetilde{M}}$. For this, let $m \in \widetilde{M}$. We distinguish two cases according to whether $gm \in H_jE$. 
\par 
\textbf{Case 1:} $gm \notin H_jE$. Then by \eqref{lemma:pre-injective-quotient-limit-2}, we have $z(gm)=t(gm)=c(gm)$. 
\par 
\textbf{Case 2:} $gm \in H_jE$. Then  
$
gm \in H_jE \setminus H_jF$ because $gm \in g\widetilde{M} \subset \Gamma \setminus (H_jF)$. 
Thus we can write $gm= h k$ where $h \in H_j$ and $k \in E$. Note that $k \notin H_jF$ for otherwise $hk \in H_jF$ as $H_j$ is a group. As $k \in E\setminus H_jF$, we infer from the relation  \eqref{lemma:pre-injective-quotient-limit-5} that  $k \notin HF$. 
Using \eqref{lemma:pre-injective-quotient-limit-2}, \eqref{lemma:pre-injective-quotient-limit} and the fact that $u\in \Fix(H_j)$, we have 
\begin{align}
\label{lemma:pre-injective-quotient-limit-3}
z(gm)&=u(gm)= u(hk)=u(k)=x(k)\\
t(gm)&=v(gm)= v(hk)=v(k)=y(k). \nonumber 
\end{align} 
Since   $x\vert_{\Gamma \setminus (HF)}= y\vert_{\Gamma \setminus (HF)}$ and $k \notin HF$, we have $x(k)=y(k)$. It follows from  \eqref{lemma:pre-injective-quotient-limit-3} that $z(gm)=t(gm)$. 
\par 
Therefore, $z\vert_{g\widetilde{M}}= t\vert_{g\widetilde{M}}$ and we can compute  
 \begin{align*}
     \widetilde{\tau}(z)(g)  = \widetilde{\mu}((g^{-1}z)\vert_{\widetilde{M}})=\widetilde{\mu}((g^{-1}t)\vert_{\widetilde{M}})  = \widetilde{\tau}(t)(g). 
 \end{align*} 
This proves that 
\begin{align}
    \label{e:lemma-pre-2}
    \widetilde{\tau}(z)\vert_{\Gamma \setminus (H_j F\widetilde{M}) }= \widetilde{\tau}(t)\vert_{\Gamma \setminus (H_j F\widetilde{M})}. 
\end{align}
From  \eqref{e:lemma-pre-1} and \eqref{e:lemma-pre-2}, we conclude that $\widetilde{\tau}(z)=\widetilde{\tau}(t)$. Since $\Psi_j(z)$ and $\Psi_j(t)$ are distinct asymptotic configurations, it follows that $\tau_j$ is not pre-injective. We thus obtain a contradiction to the pre-injectivity hypothesis of $\tau_j$. Therefore, we conclude that $\tau$ must be pre-injective and (i) is proved. 
\par 
The point (ii) is an immediate consequence of e.g. \cite[Theorem~3.4.4]{csc-book},  \cite[Proposition B.4.3]{csc-book}, and \cite[Proposition B.4.6]{csc-book}. The detailed argument is also contained in the proof of \cite[Theorem~3.7.1]{csc-book}. The proof of the lemma is thus complete. 
\end{proof}

 Our next result extends Gromov's injectivity lemma to a similar result for  post-surjectivity.   

\begin{lemma}
     \label{l:post-surjective-mark-group} Let $(H_i)_{i \in I}$ be a net in $\NN(\Gamma)$ of normal subgroups of a group $\Gamma$ which converges to some $H \in \NN(\Gamma)$. Let  $\widetilde{\tau}\colon A^\Gamma \to A^\Gamma$ be a cellular automaton. Let $\tau_i \colon A^{\Gamma/H_i} \to A^{\Gamma/H_i}$ for $i\in I$ and let $\tau \colon A^{\Gamma/H} \to A^{\Gamma/H}$  be the  cellular automata induced by $\widetilde{\tau}$. 
    Then the following  hold. 
    \begin{enumerate}[\rm (i)]
        \item  If $\tau$ is  post-surjective then  there exists $i_0 \in I$  such that  $\tau_i$ is post-surjective for every $i \in I$ with $i\geq i_0$;  
         \item  If $\tau$ is  injective then there exists $i_0 \in I$  such that  $\tau_i$ is injective for every $i \in I$ with $i\geq i_0$. 
    \end{enumerate}

\end{lemma}

\begin{proof}
(ii) is an immediate consequence of Gromov's injectivity lemma. For (i), 
we denote   $G=\Gamma/H$.  Let $Z=\Fix(H)\subset A^\Gamma$ and  $Z_i=\Fix(H_i)\subset A^\Gamma$ for all $i \in I$. Let $\pi\colon \Gamma \to G$ and $\pi_i \colon \Gamma \to \Gamma/H_i$ for $i \in I$ be the projections. 
Let $\widetilde{M}\subset \Gamma$ be a finite memory set of $\widetilde{\tau}$ and let $\widetilde{\mu} \colon A^M \to A$ be the corresponding local transition map. We can assume that  $1_\Gamma\in \widetilde{M}$ and  $\widetilde{M}=(\widetilde{M})^{-1}$. Let $M= \pi(\widetilde{M})\subset G$. Then $M$ is a memory set of $\tau$ and the associated local transition map $\mu \colon A^M \to A$ is given by $\mu(x)=\widetilde{\mu}(\Phi(x))$ for all $x \in A^{G}$ where $\Psi\colon A^K \to \Fix(H)$ is the canonical bijection described in \eqref{map:phi}. 
\par 
Suppose that  $\tau\colon A^G \to A^G$ is  a post-surjective cellular automaton. We infer from the uniform post-surjectivity Theorem~\ref{t:uniform-post-surj} that there exists a finite subset $T\subset G$ such that for every $x,y \in A^G$ with $y\vert_{G\setminus \{1_G\}} = \tau(x)\vert_{G\setminus \{1_G\}}$, there exists some configuration $z\in A^\Gamma$ satisfying   $z\vert_{G\setminus T}= x\vert_{G\setminus T} $ and $\tau(z)= y$.   Up to enlarging $T$, we can assume without loss of generality that $T$ is symmetric, that is, $T=T^{-1}$. 
We lift each $g\in T\setminus M$ to an element $\rho(g)\in \Gamma$ such that $\pi(\rho(g))=g$. 
Let us denote  
$$\widetilde{T}=\{\rho(g)\colon g\in T\}\subset \Gamma, \qquad    E = \widetilde{T}\cup \widetilde{M}.$$ 
Since $\widetilde{T}, \widetilde{M}$ are symmetric, so is $E$ and thus $1_\Gamma\in E = E^{-1}$. It follows that   
\begin{align}
\label{e:mark-group-proof-1}
    \widetilde{T},    \widetilde{M}  \subset E \subset E^2. 
\end{align}
\par 
Let $W= \{(x,y)\in A^\Gamma\times A^\Gamma \colon x\vert_{E^6} = y\vert_{E^6}  \}$ be an entourage of $A^\Gamma$. By the convergence of the net $(H_i)_{i \in I}$ to $H$ in $\NN(\Gamma)$, the net $(Z_i)_{i \in I}$ also converges  to $Z$ in the Hausdorff-Bourbaki topology on $\PP(A^\Gamma)$ (see e.g. \cite[Theorem~3.4.4]{csc-book}). 
 Consequently, there exists $i_0\in I$ such that for every $i \geq 0 $ with $i\in I$, we have $(Z,Z_i)\in \widehat{W}$ and $(H,H_i)\in V_{E^6}$, that is,  
\begin{align}
\label{e:mark-group-proof-00}
 Z\subset W[Z_i], \qquad  Z_i\subset W[Z], \qquad H\cap E^6= H_i\cap E^6. 
 \end{align}
 \par 
\textbf{Claim:} $\tau_i$ is post-surjective for every $i \geq i_0$ with $i\in I$. Indeed, fix $i \geq i_0$ and let $x_i, y_i\in Z_i$ and $u_i=\widetilde{\tau}(x_i)$ such that 
\begin{align}
    \label{e:mark-group-proof-10}
y_i\vert_{\Gamma\setminus H_i}= u_i\vert_{\Gamma\setminus H_i}.
\end{align}
Since $Z_i \in W[Z]$, there exists $x\in Z$ such that \begin{align}
    \label{e:mark-group-proof-0}
    x\vert_{E^6}=x_i\vert_{E^6}.
    \end{align}
Let $u=\widetilde{\tau}(x)$. As $\widetilde{M}$ is a memory set of $\widetilde{\tau}$, it follows from the relation $E^5\widetilde{M}\subset E^6$ (see \eqref{e:mark-group-proof-1})  that 
\begin{align}
    \label{e:mark-group-proof-2}
u\vert_{E^5}= u_i\vert_{E^5}. 
\end{align}
Let us define a configuration $y\in Z$ by setting 
\begin{align}
     \label{e:mark-group-proof-7}
     \begin{cases}
     y(h) = y_i(1_\Gamma) \text{ for all } h \in H, \\
     y\vert_{\Gamma\setminus H}= u\vert_{\Gamma\setminus H}.
     \end{cases}
\end{align} 
By the post-surjectivity of $\tau$ (or equivalently, the corresponding notion of post-surjectivity of the restriction  $\widetilde{\tau}\vert_{Z} \colon Z\to Z$) and the choice of $\widetilde{T}, T$, we can find $z \in Z$ such that 
\begin{align}
    \label{e:mark-group-proof-3}
    z\vert_{\Gamma \setminus (H\widetilde{T})} = x\vert_{\Gamma \setminus (H\widetilde{T})}, \qquad   \widetilde{\tau}(z)= y. 
\end{align}
Using the inclusion $Z\subset W[Z_i]$, we obtain a configuration $z_i\in Z_i$ such that 
\begin{align}
    \label{e:mark-group-proof-4}
    z_i\vert_{E^6}=z\vert_{E^6}. 
\end{align}
We deduce from the relations \eqref{e:mark-group-proof-4} and $E^5\widetilde{M}\subset E^6$ (by using \eqref{e:mark-group-proof-1}) that 
\begin{align} 
  \label{e:mark-group-proof-5}
\widetilde{\tau}(z_i)\vert_{E^5} = \widetilde{\tau}(z)\vert_{E^5} = y\vert_{E^5}. 
\end{align}
Let us define a configuration $w\in A^\Gamma$ by setting 
\begin{align}
\label{e:mark-group-proof-6}
\begin{cases}
   w\vert_{H_iE^3} = z_i\vert_{H_iE^3},
    \\
    w\vert_{\Gamma \setminus (H_iE^3)} =  x_i\vert_{\Gamma \setminus (H_iE^3)}. 
\end{cases}
\end{align}
Note that $w\in Z_i=\Fix(H_i)$  since $x_i, z_i \in \Fix(H_i)$. We claim that 
\begin{align}
\label{e:mark-group-proof-8}
w\vert_{H_iE^6} = z_i\vert_{H_iE^6}. 
\end{align}
Indeed, let $hg \in H_iE^6$ with $h \in H_i$ and $g \in E^6$. To prove $w(hg)=z_i(hg)$, we distinguish two cases according to whether $g\in H_iE^3$. 
\par 
\textbf{Case 1:} $g = h_i k\in H_iE^3$ where $h_i \in H_i$ and $k \in E^3$. Then we have $hh_i \in H_i$ and it follows from  \eqref{e:mark-group-proof-6} that 
$$
w(hg)= w((hh_i)k) = z_i((hh_i)k) =  z_i(hg). 
$$
\par 
\textbf{Case 2:} $g\notin H_i E^3$. Then $g\notin HE^3$. Indeed,  we can write otherwise $g=h'g'$, where $h'\in H$ and $g'\in E^3$. Thus $h'=g(g')^{-1}\in E^6$ (recall that $E^{-1}=E$) and $h'\in H \cap E^6=H_i\cap E^6 \subset H_i$ because of the relation \eqref{e:mark-group-proof-00}. Therefore, $g=h'g'\in H_iE^3$, which is a contradiction. Now as $g\notin HE^3$ and $\widetilde{T}\subset E^3$, we deduce that $g\in \Gamma \setminus (H\widetilde{T})$. It follows that 
\begin{align*}
    w(hg) & = w(g) & \text{(as } w \in \Fix(H_i) \text{ and }h\in H_i)\\
    & = x_i(g) & \text{(by }  \eqref{e:mark-group-proof-6} \text{ and } g\notin H_iE^3)\\
    & = x(g)  & \text{(by }  \eqref{e:mark-group-proof-0} \text{ and } g \in E^6)\\
    & = z(g) & \text{(by }  \eqref{e:mark-group-proof-3} \text{ and } g\in \Gamma \setminus (H\widetilde{T}) )\\
    & = z_i(g) & \text{(by }  \eqref{e:mark-group-proof-4} \text{ and } g \in E^6)\\
    & = z_i(hg) & \text{(as }   z_i \in \Fix(H_i) \text{ and } h\in H_i). 
\end{align*}
The claim~\eqref{e:mark-group-proof-8} is thus proved. The post-surjectivity of $\tau_i$ is an immediate consequence of the following claim combined with Lemma~\ref{l:post-surj}.  
\par 
\textbf{Claim:} $\widetilde{\tau}(w)=y_i$. For this, suppose first that  $hg \in H_iE^4$ where $h \in H_i$ and $g \in E^4$. Then we have $g\widetilde{M} \subset E^4\widetilde{M}\subset E^5\subset E^6$ and thus 
\begin{align*} 
\widetilde{\tau}(w)(hg)& =\widetilde{\tau}(w)(g) &  \text{(as }   \widetilde{\tau}(w) \in \Fix(H_i) \text{ and } h\in H_i)\\
& = \widetilde{\mu}((g^{-1}w)\vert_{\widetilde{M}}) & \\
&= \widetilde{\mu}((g^{-1}z_i)\vert_{\widetilde{M}}) &  \text{(by }  \eqref{e:mark-group-proof-8} \text{ and } g\widetilde{M} \subset E^6) \\
& = \widetilde{\tau}(z_i)(g) &   
\\& = y(g) & \text{(by } \eqref{e:mark-group-proof-5} \text{ and } g \in E^4\subset E^5).  
\end{align*}
If $g\in H$ then $g\in H\cap E^4=H_i\cap E^4\subset H_i$ (by \eqref{e:mark-group-proof-00} and $E^4\subset E^6$). From  the above computation, we find that  
\begin{align*}
\widetilde{\tau}(w)(hg)=y(g) & =y_i(1_\Gamma) & \text{(by }  \eqref{e:mark-group-proof-7} \text{ and } g \in H)  \\
& = y_i(hg) & \text{(as } y_i \in \Fix(H_i) \text{ and } hg \in H_i)
\end{align*} 
If $g\in E^4\setminus H$ then  $g\notin H$   and $g\notin H_i$ (because $H_i\cap E^4=H\cap E^4$) so that 
\begin{align*}
\widetilde{\tau}(w)(hg) = y(g) & = u(g)  & \text{(by }  \eqref{e:mark-group-proof-7} \text{ and } g\notin H)\\
& = u_i(g) & \text{(by }  \eqref{e:mark-group-proof-2}  \text{ and } g\in E^4\subset E^5)\\
& = y_i(g) &  \text{(by }  \eqref{e:mark-group-proof-10} \text{ and } g\notin H_i) \\
& = y_i(hg) & \text{(as } y_i \in \Fix(H_i) \text{ and } h \in H_i). 
\end{align*}
Therefore, $\widetilde{\tau}(w)\vert_{H_iE^4}= y_i\vert_{H_iE^4}$. Now let $g \in \Gamma \setminus H_iE^4$. Then $g\widetilde{M}\subset \Gamma  \setminus  (H_iE^3)$ since  $(\widetilde{M})^{-1}=\widetilde{M}\subset E$. It follows from \eqref{e:mark-group-proof-6} that $(g^{-1}w)\vert_{\widetilde{M}}=(g^{-1}x_i)\vert_{\widetilde{M}}$. On the other hand, $g\notin H_i$ (since $g\notin H_iE^4$ and $1_\Gamma\in E$).  Consequently, we obtain from \eqref{e:mark-group-proof-10} that 
\begin{align*} 
\widetilde{\tau}(w)(g)  = \widetilde{\mu}((g^{-1}w)\vert_{\widetilde{M}}) = \widetilde{\mu}((g^{-1}x_i)\vert_{\widetilde{M}}) = \widetilde{\tau}(x_i)(g) = u_i(g)= y_i(g).  
\end{align*}
This implies $\widetilde{\tau}(w)\vert_{\Gamma \setminus (H_i E^4)} = y_i\vert_{\Gamma \setminus (H_i E^4)}$. We conclude that $\widetilde{\tau}(w)=y_i$. Therefore, ${\tau}_i$ is post-surjective for every $i \geq i_0$ with $i\in I$. This shows (i) and the proof of the theorem is complete. 
\end{proof}

As a concluding remark of this   section, the following simple example shows that the converse of each of the statements (i) and (ii) in the above Lemma~\ref{l:post-surjective-mark-group} is not true for both post-surjectivity and injectivity. In fact, the example also shows that the limit of a sequence of invertible cellular automata may be a non-invertible cellular automaton. 

\begin{example}
Consider the binary field alphabet $A = \{0,1\}$. Let $M= \{0,1,2\}$ be a subset of the group universe $G= \Z$. We define a cellular automaton $\tau\colon A^\Z \to A^\Z$ admitting $M$ as a memory set. The local transition map   $\mu\colon A^M \to A$ is given by  the formula  $\mu(x,y,z)=x+y+z \in A$ (mod 2). For each integer $n \geq 1$, let $H_n= (3n+1)\Z$ be the kernel of the projection map $\pi_n \colon \Z \to \Z/H_n$. It is clear that the sequence $(H_n)_{n \geq 1}$ converges to the trivial subgroup $H=\{1_G\} \subset \Z$. 
\par 
Let $Z_n=\Fix(H_n)\subset A^\Z$ be the subshift of $H_n$-periodic configurations. We denote by   $\tau_n \colon A^{\Z/H_n} \to A^{\Z/H_n}$ the cellular automaton conjugate to the restriction $\tau\vert_{Z_n}\colon Z_n \to Z_n$.  Observe from our construction that $\tau_n$,  as a linear endomorphism of the vector space $A^{\Z/H_n}= A^{\{\overline{0}, \overline{1}, \dots, \overline{3n}\}}\simeq A^{3n+1}$ is given by the following $(3n+1)\times (3n+1)$ circulant matrix 
$$
C_n= \left(\begin{matrix}
    1 & 1 & 1 & 0 & 0 & 0 & \dots & 0 & 0 & 0 \\
    0 & 1 & 1 & 1 & 0 & 0 & \dots & 0 & 0 & 0 \\
    0 & 0 & 1 & 1 & 1 & 0 & \dots & 0 & 0 & 0 \\
     &  & \dots &  &  &  & \dots &  & \dots  &  \\
     0 & 0 & 0 & 0 & 0 & 0 & \dots & 1 & 1 & 1 \\
     1 & 0 & 0 & 0 & 0 & 0 & \dots & 0 & 1 & 1 \\
     1 & 1 & 0 & 0 & 0 & 0 & \dots & 0 & 0 & 1 
\end{matrix}
\right)
$$
Let $f(x)= 1+x+x^2$ be the associated polynomial of $C$ and let $\omega_n = \exp\left(\frac{2i\pi}{3n+1}\right)\subset \C$ be a primitive $(3n+1)$-th root of $1$. Note that $f(1)=3$ and 
$$f(\omega_n^k)= \frac{\omega_n^{3k}-1}{\omega^k_n-1}$$ for all $k=1,\dots, 3n$ since $\omega_n^k\neq 1$. Moreover,  as $\mathrm{gcd}(3,3n+1)=1$, we have $$\{\omega_n^{3k}\colon k=1,\dots,3n\}= \{\omega_n^{k}\colon k=1,\dots,3n\}.$$ 
The real determinant $D_\R(C_n)\in \R$ of $C_n$, where we regard $C_n$ as a real circulant matrix, can be computed by the formula (see e.g. \cite{gray}): 
\begin{align*}
D_\R (C_n )& = \prod_{k=0}^{3n}f(\omega_n^k)  = f(\omega_n^0) \prod_{k=1}^{3n}f(\omega_n^k)\\
&= 3  \prod_{k=1}^{3n} \frac{\omega_n^{3k}-1}{\omega^k_n-1} = 3   \frac{\prod_{k=1}^{3n} (\omega_n^{3k}-1)}{\prod_{k=1}^{3n} (\omega^k_n-1)}\\
 &= 3   \frac{\prod_{k=1}^{3n} (\omega_n^{k}-1)}{\prod_{k=1}^{3n} (\omega^k_n-1)}=3. 
\end{align*}
Consequently, the determinant $D_A(C_n)\in A$ of $C_n$ (where $C_n$ is regarded as a matrix with coefficients in the binary field $A$) is 
$$D_A(C_n)= D_\R(C_n) \text{ mod }2 = 1\neq 0. 
$$
Therefore, every cellular automaton $\tau_n$ is a linear isomorphism of the finite dimensional vector space $A^{\Z/H_n}$. The sequence $(\tau_n)_{n \geq 1}$ thus consists of  invertible (and also injective post-surjective) cellular automata. 
Note that for one-dimensional cellular automata with finite alphabet (the group universe being $\Z$), injectivity and  post-surjectivity are equivalent notions and they are both equivalent to invertibility.      However,  the limit cellular automaton $\tau$ of the sequence of cellular automata $(\tau_n)_{n \geq 1}$  is not invertible since $\tau$ is not injective:  $\tau(c)=\tau(d)=c$ where $c,d\in A^\Z$ are distinct configuration defined for all $k \in \Z$ by   
    $$
    c(k)=0, \quad d(3k)=0, \quad d(3k+1)=d(3k+2)=1. 
    $$
\end{example}

\section{Closedness of marked post-injunctive groups}
\label{s:closedness-mark-groups}

In this section, we show that post-injunctive groups form a closed subset of the space of $\Gamma$-marked groups for any group $\Gamma$.

\begin{theorem}
\label{t:marked-groups-closed}
Let $\Gamma$ be a group. Then the set of normal subgroups $H \subset \Gamma$ such that the quotient  $\Gamma/H$ is a post-injunctive group is closed in $\NN(\Gamma)$. 
\end{theorem}

\begin{proof}
Let $(H_i)_{i \in I}$ be a net in $\NN(\Gamma)$ such that the quotient groups $\Gamma/H_i$ are post-injunctive. Suppose that the net $(H_i)_{i \in I}$ converges to some normal subgroup $H \in \PP(\Gamma)$. Let $Z=\Fix(H)\subset A^\Gamma$ and  $Z_i=\Fix(H_i)\subset A^\Gamma$ for all $i \in I$.  We have to show that the quotient group $G=\Gamma/H$ is  post-injunctive. 
\par 
Let $\pi \colon \Gamma \to G$ be the induced projection. Let $A$ be a finite alphabet and let $\tau\colon A^G \to A^G$ be a post-surjective group. 
Let $M\subset G$ be a finite memory set of $\tau$ and let $\mu\colon A^M \to A$ be the corresponding local transition map. We fix an element $\rho(g)\in \Gamma$ for each $g\in M$ such that $\pi(\rho(g))=g$. Let 
$$\widetilde{M} = \{\rho(g)\colon g \in M\}\subset \Gamma.  
$$
The map $\mu$ lifts to a map $\widetilde{\mu} \colon A^{\widetilde{M}} \to A$ via the bijection  $\rho \colon  A^{\widetilde{M}}\to  A^M$
 defined by $\rho(x)(g)= x(\pi(g))$ for all $x \in A^{\widetilde{M}}$ and $g \in M$. Let $\widetilde{\tau}\colon A^\Gamma \to A^\Gamma$ be the cellular automaton with memory $\widetilde{M}$ and local transition map $\widetilde{\mu}$. Note that the restriction $\widetilde{\tau}\vert_{Z}\colon Z \to Z$ is conjugate to the cellular automaton $\tau\colon A^G\to A^G$ and similarly, every restriction $\widetilde{\tau}\vert_{Z_i}\colon Z_i \to Z_i$ is conjugate to some induced cellular automaton $\tau_i \colon A^{\Gamma/H_i}\to A^{\Gamma/H_i}$. 
 \par
Since $\tau$ is post-surjective by hypothesis,  we infer from Lemma~\ref{l:post-surjective-mark-group} that there exists $i_0\in I$ such that ${\tau}_i$ is post-surjective for every $i_0 \in I$ with $i\geq i_0$.  Since  $\Gamma/H_i$ is post-injunctive for every $i \in I$, it follows that $\tau_i$ must be pre-injective for all $i \geq i_0$. Since the net $(H_i)_{i\in I, i\geq 0}$ converges to $H$ in $\NN(\Gamma)$, we deduce from Lemma~\ref{l:pre-injective-mark-group} that the cellular automaton $\tau$ must be pre-injective as well. This shows that $G$ is a post-injunctive group and the proof of the theorem is complete. 
\end{proof}

\bibliographystyle{splncs04}

\begin{thebibliography}{8} 


\bibitem{arzhantseva}
{G. Arzhantseva, S. R. Gal.}: {On approximation properties of semidirect products of groups}, Annales mathématiques Blaise Pascal. 27, 125--130 (2020). 


\bibitem{bartholdi-kielak}
{Bartholdi, L.}:  
{Amenability of groups is characterized by Myhill's Theorem. With an appendix by D. Kielak}, 
J. Eur. Math. Soc. vol. 21, Issue 10 (2019), pp. 3191--3197. 
  
 

 

\bibitem{burks}
{Burks, A.W.}:
von Neumann’s self-reproducing automata. In: Burks, A.W. (ed.) Essays on Cellular Automata, pp. 3–64. University of Illinois Press, Champaign (1971)

 


\bibitem{kari-post-surjective}
{Capobianco, S., Kari, J., Taati, S.}: 
{An “almost dual” to Gottschalk’s Conjecture}. 22th International Workshop on Cellular Automata and Discrete Complex Systems (AUTOMATA), Jun 2016, Zurich, Switzerland. pp.77--89

\bibitem{tullio}
{\sc T.~Ceccherini-Silberstein, A.~Mach{\`{\i}}, and F.~Scarabotti}, {\em
  Amenable groups and cellular automata}, Ann. Inst. Fourier (Grenoble), 49
  (1999), pp.~673--685.
   
\bibitem{hedlund-csc}
{Ceccherini-Silberstein, T.,  Coornaert, M.}:  
A generalization of the Curtis-Hedlund theorem, Theoret. Comput. Sci., 400 (2008), pp. 225–229


\bibitem{csc-book} 
{Ceccherini-Silberstein, T.,  Coornaert, M.}: 
{Cellular Automata and Groups},  Springer Monographs in Mathematics, Springer-Verlag, Berlin, 2010. 


 
 

\bibitem{ceccherini}
{Ceccherini-Silberstein, T., Mach{\`{\i}}, A., Scarabotti, F.}: {
  Amenable groups and cellular automata}, Ann. Inst. Fourier (Grenoble), 49
  (1999), pp.~673--685.

 
 


 
\bibitem{Den-12a}  
{Dennunzio, A., Formenti, E., Provillard, J.}:   {Non-uniform cellular automata: Classes, dynamics, and decidability},  
Information and Computation
Volume 215, June 2012, Pages 32-46
 
\bibitem{Den-12b} 
{Dennunzio, A., Formenti, E., Provillard, J.}:    {Local rule distributions, language complexity and non-uniform
cellular automata}. Theoretical Computer Science, 504  (2013) 38–51.  
 

  



\bibitem{gromov-esav}
{Gromov, M.}:  {Endomorphisms of symbolic algebraic varieties}, J. Eur.
  Math. Soc. (JEMS), 1 (1999), pp.~109--197.
 



\bibitem{gottschalk}
{Gottschalk, W.H.}:  
{Some general dynamical notions}, 
Recent advances in topological dynamics, Springer, Berlin, 1973, pp. 120--125. Lecture Notes in Math. Vol. 318.

\bibitem{gray}
{Gray, R.M.}: 
{Toeplitz and circulant matrices: A review},  Foundations and Trends in Communications and Information Theory (2006), 2 (3): 155--239. 


\bibitem{hedlund}  {Hedlund, G.A.}: 
{Endomorphisms and automorphisms of the shift dynamical system}, 
Math. Systems Theory, 3 (1969), pp.~320--375.


  
  



\bibitem{moore}
{Moore, E.F.}:  {Machine models of self-reproduction}, vol.~14 of Proc.
  Symp. Appl. Math., American Mathematical Society, Providence, 1963,
  pp.~17--34.

\bibitem{myhill}
{Myhill, J.}:  {The converse of {M}oore's {G}arden-of-{E}den theorem},
  Proc. Amer. Math. Soc., 14 (1963), pp.~685--686.
 

 



\bibitem{neumann}
{von Neumann, J.}:  
In: Burks, A.W. (ed.) The Theory of Self-reproducing Automata. University of Illinois Press, Urbana (1966)

 

 


\bibitem{phung-tcs}
{Phung, X.K.}: 
{On invertible and stably reversible non-uniform cellular automata},  
Theoret. Comput. Sci., vol. 940 (2023), pp.~43--59, https://doi.org/10.1016/j.tcs.2022.09.011
 
 


 
  


  
 


 

  
\bibitem{weiss-sgds}
{Weiss, B.}:  
{Sofic groups and dynamical systems}, 
Sankhy\=a Ser. A 62 (2000), no. 3, pp.~350--359. 
Ergodic theory and harmonic analysis (Mumbai, 1999). 

 

\end{thebibliography}

\end{document}